\documentclass[12pt]{amsart}
\usepackage{amssymb,amscd}
\usepackage{verbatim}

\usepackage{amsmath,amssymb,graphicx,mathrsfs}   
\usepackage{enumerate}
\usepackage[colorlinks=true,allcolors = blue]{hyperref} 

\usepackage{tikz}
\usetikzlibrary{matrix}

\usepackage[all]{xy}

\let\frak\mathfrak

\def\>{\relax\ifmmode\mskip.666667\thinmuskip\relax\else\kern.111111em\fi}
\def\<{\relax\ifmmode\mskip-.333333\thinmuskip\relax\else\kern-.0555556em\fi}
\def\vsk#1>{\vskip#1\baselineskip}
\def\vv#1>{\vadjust{\vsk#1>}\ignorespaces}
\def\vvn#1>{\vadjust{\nobreak\vsk#1>\nobreak}\ignorespaces}

  \let\ssize\scriptstyle
\let\sssize\scriptscriptstyle

\let\Medskip\medskip
\def\medskip{\par\Medskip}
\let\Bigskip\bigskip
\def\bigskip{\par\Bigskip}

\let\Maketitle\maketitle
\def\maketitle{\Maketitle\thispagestyle{empty}\let\maketitle\empty}

\newtheorem{thm}{Theorem}[section]
\newtheorem{cor}[thm]{Corollary}
\newtheorem{lem}[thm]{Lemma}

\newtheorem{defn}[thm]{Definition}

\theoremstyle{definition}                                  
\numberwithin{equation}{section}

\theoremstyle{definition}
\newtheorem*{rem}{Remark}
\newtheorem*{example}{Example}

\let\mc\mathcal
\let\nc\newcommand

\let\dl\delta
\let\Dl\Delta

\let\ka\kappa
\let\la\lambda
\let\La\Lambda

\let\phi\varphi
\let\si\sigma

\let\Om\Omega

\let\der\partial

\let\ox\otimes

\let\ge\geqslant
\let\geq\geqslant
\let\le\leqslant
\let\leq\leqslant

\let\on\operatorname
\let\bi\bibitem
\let\bs\boldsymbol

\def\C{{\mathbb C}}
\def\Z{{\mathbb Z}}
\def\R{{\mathbb R}}

\def\F{{\mathbb F}}   

\def\+#1{^{\{#1\}}}

\def\beq{\begin{equation}}
\def\eeq{\end{equation}}
\def\be{\begin{equation*}}
\def\ee{\end{equation*}}

\nc{\bea}{\begin{eqnarray*}}
\nc{\eea}{\end{eqnarray*}}
\nc{\bean}{\begin{eqnarray}}
\nc{\eean}{\end{eqnarray}}

\def\g{{\mathfrak g}}

\nc{\Il}{{\mc I_{\bs\la}}}
\nc{\bla}{{\bs\la}}
\nc{\Fla}{\F_\bla}
\nc{\tfl}{{T^*\Fla}}
\nc{\GL}{{GL_n(\C)}}
\nc{\GLC}{{GL_n(\C)\times\C^*}}

\let\sd s 

\def\ddk_#1{\kk_{#1}\<\>\frac\der{\der\<\>\kk_{#1}}}

\def\bul{\mathbin{\raise.2ex\hbox{$\sssize\bullet$}}}
\def\intt{\mathchoice
{\mathop{\raise.2ex\rlap{$\,\,\ssize\backslash$}{\intop}}\nolimits}
{\mathop{\raise.3ex\rlap{$\,\sssize\backslash$}{\intop}}\nolimits}
{\mathop{\raise.1ex\rlap{$\sssize\>\backslash$}{\intop}}\nolimits}
{\mathop{\rlap{$\sssize\<\>\backslash$}{\intop}}\nolimits}}

\let\kk q 
\let\cc c

\let\Ko K

\def\GZ/{Gelfand-Zetlin}
\def\KZ/{{\slshape KZ\/}}
\def\qKZ/{{\slshape qKZ\/}}
\def\XXX/{{\slshape XXX\/}}

\nc{\A}{{\mc A}}

\def\sll{{\frak{sl}}}

\def\Q{{\mathbb Q}}

\nc{\hsl}{\widehat{{\frak{sl}_2}}}

\nc{\BC}{{ \mathbb C}}
\nc{\lra}{\longrightarrow}
\nc{\CO}{{\mathcal{O}}}
\nc{\BZ}{{ \mathbb Z}}
\nc{\hfn}{\hat{\frak{n}}}
\nc\Zs{{\Z/p^s\Z}}
\nc\Zo{{\Zs[z]^0}}
\nc\gr{{\on{gr}}}

\nc\fD{{\frak D}}

\newcommand{\Cf}{\operatorname{Coeff}}

\begin{document}

\hrule width0pt
\vsk->

\title[Congruences for solutions of $p$-adic  KZ equations]
{Congruences for Hasse--Witt matrices 
\\ and solutions of $p$-adic  KZ equations}

\author[A.\:Varchenko and W.\,Zudilin]
{Alexander Varchenko$^{\star}$ and Wadim Zudilin$^{\diamond}$}

\maketitle

\begin{center}
{\it $^{\star}$ Department of Mathematics, University
of North Carolina at Chapel Hill\\ Chapel Hill, NC 27599-3250, USA\/}

\vsk.5>
{\it $^{ \star}$ Faculty of Mathematics and Mechanics, Lomonosov Moscow State
University\\ Leninskiye Gory 1, 119991 Moscow GSP-1, Russia\/}

\vsk.5>
 {\it $^{ \star}$ Moscow Center of Fundamental and Applied Mathematics
\\ Leninskiye Gory 1, 119991 Moscow GSP-1, Russia\/}

\vsk.5>
 {\it $^{ \diamond}$ 
Institute for Mathematics, Astrophysics and Particle Physics, Radboud University,
\\
 PO Box 9010, 6500 GL, Nijmegen, The Netherlands\/}

\end{center}

\vsk>
{\leftskip3pc \rightskip\leftskip \parindent0pt \Small
{\it Key words\/}:  KZ equations; Dwork congruences; master polynomials; Hasse--Witt matrices.

\vsk.6>
{\it 2020 Mathematics Subject Classification\/}: 11D79 (12H25, 32G34, 33C05, 33E30)
\par}

{\let\thefootnote\relax
\footnotetext{\vsk-.8>\noindent
$^\star\<${\sl E\>-mail}:\enspace anv@email.unc.edu,
supported in part by NSF grant DMS-1954266
\\
$^\diamond\<${\sl E\>-mail}:\enspace  w.zudilin@math.ru.nl
}}

\begin{abstract}

We prove general Dwork-type congruences for Hasse--Witt matrices
attached to tuples of Laurent polynomials.
We apply this result to establishing arithmetic and $p$-adic analytic properties of functions
 originating from polynomial solutions modulo $p^s$ of Knizhnik--Zamolodchikov (KZ) equations, 
 solutions which come as coefficients of master polynomials and whose coefficients are integers. 
As an application we show that the  $p$-adic KZ connection associated with the family of hyperelliptic curves
$y^2=(t-z_1)\dots (t-z_{2g+1})$ has an invariant  subbundle of rank $g$. Notice that 
the corresponding complex KZ connection
 has no  nontrivial subbundles due to the irreducibility  of its monodromy representation.

\end{abstract}

{\small\tableofcontents\par}

\setcounter{footnote}{0}
\renewcommand{\thefootnote}{\arabic{footnote}}

\section{Introduction}

It is classical that the periods of the Legendre family $y^2=t(t-1)(t-x)$ of elliptic curves viewed as functions of $x$ satisfy the hypergeometric differential equation
\bean
\label{HE}
x(1-x) I'' +(1-2x)I'-\frac14I=0;
\eean
the hypergeometric series
\bean
\label{2F1}
\phantom{aaaa}
F(x):={}_2F_1\Big(\frac12,\frac 12; 1; x\Big) = \frac 1\pi\,\int_1^\infty t^{-1/2}(t-1)^{-1/2}(t-x)^{-1/2} dt
=\sum_{k = 0}^\infty\binom{-1/2}{k}^2x^k
\eean
represents the analytic solution of \eqref{HE} at the origin.
In order to investigate the local zeta function of the $x$-fiber in the family, Dwork \cite{Dw} studied
 the differential equation \eqref{HE} $p$-adically using  truncations of the infinite sum in \eqref{2F1},
\bea
F_s(x)=\sum_{k = 0}^{p^s-1}\binom{-1/2}{k}^2x^k \qquad\text{for}\quad s=1,2,\dots,
\eea
as $p$-adic approximations to its  analytic solution \eqref{2F1}.
Clearly, the series $(F_s(x))$ converges to $F(x)$ as $s\to\infty$ in the disk $D_{0,1}=\{x\mid |x|_p<1\}$.
Dwork showed that the uniform limit $F_{s+1}(x)/F_s(x^p)$ as $s\to\infty$ exists in a larger domain 
$\frak D_{\text{Dw}}$, thus giving the $p$-adic analytic continuation of the function
$F(x)/F(x^p)$, originally defined in $D_{0,1}$, to that larger domain.
This limit, called the ``unit root'', defines a root of the local zeta function.

The second part of Dwork's investigation \cite{Dw} concerned with
 the $p$-adic analytic continuation of the function $F(x)$ as a solution
of \eqref{HE}. Dwork considered differential equation \eqref{HE} as a system of 
 first order linear differential equations
for the vector $(F(x), F'(x))$ and approximated the direction vector 
$(1, F'(x)/F(x))$ by rational functions $(1, F'_s(x)/F_s(x))$. He showed that  
the uniform limit   as $s\to\infty$ of these rational functions does 
exist in the same larger domain 
$\frak D_{\text{Dw}}$, thus giving the $p$-adic analytic continuation to the domain 
$\frak D_{\text{Dw}}$ of the direction vector $(1, F'(x)/F(x))$
(but not of   the solution  $(F(x), F'(x))$).

This fact indicates  a clear difference of structure between solutions of complex analytic linear differential equations
and their $p$-adic versions.
A local solution of a complex linear differential equation can be analytically continued to a multi-valued
analytic function on the domain of the definition of the differential equation, while in the $p$-adic setting only
certain  subspaces of the space of all local solutions can be analytically continued as subspaces to larger
domains. In Dwork's situation only the one-dimensional subspace generated by  $F(x)$ in the two-dimensional
space of all local solutions at $x=0$ can be $p$-adic
 analytically continued to the larger domain $\frak D_{\text{Dw}}$. 
 
\vsk.2>

Dwork's work initiated significant research in $p$-adic differential equations and their applications to arithmetic of algebraic varieties. There is hardly a way to list all of them here, so we limit ourselves to 
mentioning some very recent contributions on the theme \cite{AS, BV}.  The principal direction of that research 
is generalization of the first part of Dwork's work\,---\,on the relation between the unit root $F(x)/F(x^p)$ and the
zeros of the  
local zeta function of the $x$-fiber in the Legendre family $y^2=t(t-1)(t-x)$.
In such a generalization the function $F(x)/F(x^p)$ becomes a square matrix with roots of its characteristic polynomial
related to zeros of the local zeta function of the fibers of the
corresponding family of algebraic varieties.

\vsk.2> 

Our present paper is related to the second part of Dwork's investigation on the $p$-adic analytic continuation of
the direction vector $(1, F'(x)/F(x))$. We study this phenomenon for a system of Knizhnik--Zamolodchikov (KZ)
differential equations. The KZ equations over $\C$ are objects 
of conformal field theory, representation theory, enumerative geometry, see for example \cite{KZ, EFK, MO}.
In \cite{SV1} the KZ equations over $\C$ were identified with the differential equations for flat sections of a suitable 
Gauss--Manin 
connection and solutions of the KZ differential equations were constructed in the form of multidimensional
hypergeometric integrals. In that sense the KZ differential equations are distant relatives of the hypergeometric
differential equation
\eqref{HE}. It is known that the KZ equations and their solutions have remarkable properties, see for example
\cite{EFK, V2}. This motivates the study of KZ differential equations and their solutions
over $p$-adic fields.

\vsk.2>

We consider the  differential KZ equations over $\C$  in the special case, when the hypergeometric
solutions are given by  hyperelliptic integrals of genus $g$. 
In this case the space of solutions of the differential KZ equations is a $2g$-dimensional complex vector space.
We also consider the $p$-adic version of the same differential
equations.  We show that the $2g$-dimensional space of local solutions of these $p$-adic 
differential KZ equations has a remarkable $g$-dimensional subspace of solutions which can be $p$-adic analytically
continued as a subspace to a large domain $\frak D_{\on{KZ}}^{(m),o}$
in the space where the KZ equations are defined,   see Theorems \ref{thm inv} and \ref{thm rk g} for precise statements.
This $g$-dimensional global subspace of solutions is defined constructively as the uniform $p$-adic limit of the
$g$-dimensional subspace of polynomial solutions of these KZ equations modulo $p^s$, the polynomial
solutions constructed in \cite{V5}.

\vsk.2> 

In \cite{SV2} general KZ differential equations were considered over the field $\F_p$ and their polynomial solutions
were constructed. 
In \cite{V5} this construction was modified and polynomial solutions modulo $p^s$ 
of the KZ equations associated with the hyperelliptic integrals were constructed. 
The polynomial solutions are vectors of polynomials with integer coefficients. They 
are  ``$p^s$-approximations'' of the corresponding hyperelliptic integral solutions of the same differential 
KZ equations over $\C$. In \cite{V5} the constructed polynomial solutions are called the  
$p^s$-hypergeometric solutions.  
While the complex hyperelliptic  integrals give the $2g$-dimensional space of all solutions of the 
complex KZ equations, the $p^s$-hypergeometric  solutions span only a $g$-dimensional subspace.
More general $p^s$-approximation 
constructions are discussed in \cite{SV2, RV1, RV2}.

\vsk.2>

In order to prove his two results stated above, Dwork developed in \cite{Dw} two types of congruences,
\bea
F_{s+1}(x)/F_s(x^p)
&\equiv&
  F_{s}(x)/F_{s-1}(x^p) \pmod{p^s}\,,
\\
F_{s+1}'(x)/F_{s+1}(x)
&\equiv&
  F_{s}'(x)/F_{s}(x) \quad\ \pmod{p^s}\,,
\eea
which are now called Dwork congruences.  A suitable matrix form of Dwork congruences is used in
most papers on $p$-adic differential equations. The closest versions of Dwork congruences
related to our needs  were developed in papers 
 \cite{Me, MV, Vl} by Mellit and Vlasenko. 
  Motivated by our KZ equation considerations  we give a 
   generalization of  Dwork congruences from \cite{Me, MV, Vl} in  Section \ref{sec 2}, see Theorems \ref{thm 1.6} and \ref{thm der}.
   The proofs of these theorems are  modifications of the proofs in \cite{Vl}. 
 
\vsk.2>
 In Sections \ref{sec 5} and \ref{sec 6} we apply  Dwork congruences from Section \ref{sec 2} to the matrices
 composed of coordinates of the   $p^s$-hypergeometric solutions and their antiderivatives.
 This application allows us to define the $g$-dimensional global subspace of solutions 
 on a large domain $\frak D_{\on{KZ}}^{(m),o}$.

 \vsk.2>
 
 Notice that the main tool in \cite{Dw} to prove the properties of the function
 $F(x)$ are polynomials $F_s(x)$ which are truncations of the power series expansion of the function $F(x)$. 
 In our KZ equation case we do not have
 distinguished solutions whose power series expansions  may be truncated and whose
 ratios could be  $p$-adic analytically continued. Instead, we have a collection of
$p^s$-hypergeometric solutions defined independently of any Dwork congruences, but 
which surprisingly satisfy appropriate Dwork congruences and give us  a global subspace of solutions
in the $s\to\infty$ limit.
 
\vsk.2>

This paper may be viewed as a continuation of our work \cite{VZ} where the case $g=1$ is developed.

\subsection*{Acknowledgements}
The authors thank Jeff Achter,  Vladimir Berkovich, Frits Beukers and Masha Vlasenko
 for useful discussions.

\section{On ghosts}
\label{sec 2}

In this paper $p$ is an odd prime.  We denote by $\Z_p[w^{\pm1}]$ the ring of Laurent polynomials in variables $w$
with coefficients in $\Z_p$.
A congruence $F(w)\equiv G(w)\pmod{p^s}$ for two Laurent polynomials from the ring is understood as the divisibility by $p^s$ of all coefficients of $F(w)-\nobreak G(w)$.

Throughout this section $x=(t,z)$, where $t=(t_1,\dots,t_r)$ and $z=(z_1,\dots,z_n)$ are two groups of variables.

\subsection{Definition of ghosts}

Let $\La=(\La_0(x),\La_1(x), \dots, \La_l(x))$ be a tuple of Laurent polynomials in $\Z_p[x^{\pm1}]$.
For every $0\leq j\leq s\leq l$, we define the Laurent polynomials
$$
W_s(x)=W_s(\La_0,\dots,\La_s)(x):=\La_0(x)\La_1(x)^{p}\dots \La_s(x)^{p^s}
$$
and
$$
W_s^{(j)}(x)=W_s^{(j)}(\La_0,\dots,\La_s)(x):=W_{s-j}(\La_j,\dots,\La_s)(x)=\La_j(x)\La_{j+1}(x)^{p}\dots \La_s(x)^{p^{s-j}}.
$$
Furthermore, we 
introduce the tuple $V_s=V_s(x)=V_s(\La_0,\dots,\La_s)(x)$, $s=0,1,\dots,l$, 
of Laurent polynomials in $\Z_p[x^{\pm1}]$
by the recursive formula
\begin{equation}
\label{dls}
V_s(x)
=W_s(x)-\sum_{j=1}^sV_{j-1}(x)W_s^{(j)}(x^{p^j}).
\end{equation}

The Laurent polynomials  $V_s(x)$ are called {\it ghosts} associated with the tuple $\La$.
They are useful for studying the congruences related to the tuple,
but they do not enter the final results.
The ghosts $V_s(x)$ are essentially offered in Vlasenko's work \cite{Vl}, though stated there for a very particular situation.
The ghosts $V_s(x)$ 
are quite different from the ghosts we use in our previous work \cite{VZ}\,---\,those are rooted in Mellit's preprint \cite{Me}.

\begin{lem}
\label{lem dl}

For $s=0,1,\dots,l$, we have $V_s(x) \equiv 0 \pmod{p^s}$.
\end{lem}

\begin{proof} For $s=0$ we have $V_0(x)=\La_0(x)$ and no requirements on divisibility.
For $s=1$, we have modulo $p$:
\bea
V_1(x) = \La_0(x)\La_1(x)^p -V_0(x)\La_1(x^p)
= \La_0(x)(\La_1(x)^p -\La_1(x^p))\equiv 0.
\eea
More generally, for $s\ge1$ applying $V_{j-1}(x)\equiv0\pmod{p^{j-1}}$ for $0<j\le s$
and
\begin{align*}
\La_s(x)^{p^s}
&\equiv\La_s(x^p)^{p^{s-1}}\pmod{p^s}
\equiv\La_s(x^{p^2})^{p^{s-2}}\pmod{p^{s-1}}
\equiv
\\
&\equiv\dots
\equiv\La_s(x^{p^j})^{p^{s-j}}\pmod{p^{s-j+1}}
\end{align*}
(which follows from iterative use of $F(x^p)^{p^{i-1}}\equiv F(x)^{p^i}\pmod{p^i}$ valid for $i>0$) we deduce modulo $p^s$:
\begin{align*}
V_s(x)
&=W_{s-1}(x)\La_s(x)^{p^s}
-\sum_{j=1}^{s-1}V_{j-1}(x)W_{s-1}^{(j)}(x^{p^j})\La_s(x^{p^j})^{p^{s-j}}
-V_{s-1}(x)\La_s(x^{p^s})
\equiv
\\ &
\equiv\bigg(W_{s-1}(x)
-\sum_{j=1}^{s-1}V_{j-1}(x)W_{s-1}^{(j)}(x^{p^j})
-V_{s-1}(x)\bigg)\La_s(x)^{p^s}
=0,
\end{align*}
giving the required statement.
\end{proof}

For a Laurent polynomial $F(t,z)$ in $t,z$, let $N(F)\subset \R^r$ be the Newton polytope of $F(t,z)$ with respect to 
the $t$ \emph{variables only}.

\begin{lem}
\label{lem 1.2}
For $s=0,1,\dots,l$, we have
\bea
N(V_s) \subset N(\La_0)+ pN(\La_1) + \dots+p^sN(\La_s)\,.
\eea

\end{lem}

\begin{proof}
This follows from \eqref{dls} by induction on $s$.
\end{proof}

\subsection{Admissible tuples}

Let $\Dl\subset \Z^r$ be a finite subset.

\begin{defn}
\label{defN}
A tuple $(N_0,N_1,\dots,N_l)$ of convex polytopes in $\R^r$
is called $\Dl$-\emph{admissible} if for any
$0\le i\le j < l$ we have
\bea
\big(\Dl + N_i+ pN_{i+1} + \dots+p^{j-i}N_j\big)\cap p^{j-i+1}\Z^r  \subset p^{j-i+1}\Dl\,.
\eea

\end{defn}  
\vsk.2>

Notice that subtuples $(N_i,N_{i+1},\dots,N_j)$ of a $\Dl$-admissible tuple are also $\Dl$-admissible.

\begin{example}
Let $r=1$, $\Dl=\{0,1\}\subset \Z$ and $N=[-(p-1)/2, 3(p-1)/2]\subset \R$. Then
the tuple
$(N,N, \dots, N)$ is $\Dl$-admissible.

\end{example}

\vsk.2>

\begin{defn}
\label{defn}

A tuple $(\La_0(t,z),\La_1(t,z),\dots,\La_l(t,z))$ of Laurent polynomials 
is called $\Dl$-\emph{ad\-missible}
if the tuple 
$\big(N(\La_0), N(\La_1), \dots, N(\La_l)\big)$ is $\Dl$-admissible.
\end{defn}

\subsection{Hasse--Witt matrix}

For $v\in\Z^r$ denote by $\Cf_v F(t,z)$ the coefficient of $t^v$ in the Laurent polynomial $F(t,z)$;
clearly,  this is a Laurent polynomial in $z$.

\vsk.2>

Given $m\ge 1$ and a finite subset $\Dl\subset \Z^r$ with $g=\#\Dl$,
define  the $g\times g$ \emph{Hasse--Witt matrix} of level $m$ of the Laurent polynomial $F(t,z)$
by the formula
\bean
\label{Cuv}
A(m, F(t,z))
:=
\big( \Cf_{p^mv-u}F(t,z)\big)_{u\in\Dl, v\in\Dl}\,.
\eean
The entries of this matrix are Laurent polynomials in $z$.

Furthermore, for a Laurent polynomial $G(z)$ define $\si(G(z))=G(z^p)$.

\begin{lem}
\label{lem 1.5}

Let $(\La_0(t,z),\La_1(t,z),\dots,\La_l(t,z))$ be a $\Dl$-admissible tuple of Laurent polynomials in $\Z_p[x^{\pm1}]=\Z_p[t^{\pm1}, z^{\pm1}]$. 
Then for $0\le s\le l$ we have
\begin{alignat*}{2}
\textup{(i)} &\;\; &
A(s+1, V_s) &\equiv 0 \pmod{p^s};
\\
\textup{(ii)} &\;\; &
A\big(s+1, W_s\big)
&=
A\big(1,V_0) \cdot  \si\big(A\big(s, W_s^{(1)}\big)\big)
+
A\big(2,V_1) \cdot \si^2\big(A\big(s-1, W_s^{(2)}\big)\big)
+\dots 
+ 
\\ &&
&\qquad
+A\big(s,V_{s-1}) \cdot \si^s\big(A\big(1, W_s^{(s)}\big)\big)
+ A\big(s+1,V_s).
\end{alignat*}

\end{lem}

\begin{proof}
Part (i) follows from Lemma \ref{lem dl}. To prove (ii) consider the identity
\begin{align*}
&
\La_0(t,z)\La_1(t,z)^{p}\dots \La_s(t,z)^{p^s}
=
\\ &\qquad
=\sum_{j=1}^sV_{j-1}(t,z)\La_j(t^{p^j},z^{p^j})\La_{j+1}(t^{p^j},z^{p^j})^p\dots\La_s(t^{p^j},z^{p^j})^{p^{s-j}}
+ V_s(t,z),
\end{align*}
which is nothing else but \eqref{dls}.
Let $u,v\in\Dl$. In order to calculate the coefficient of $t^{p^{s+1}v-u}$
in the term $V_{j-1}(t,z)\La_j(t^{p^j},z^{p^j})\dots\La_s(t^{p^j},z^{p^j})^{p^{s-j}}$,
we look for all pairs of vectors $w\in N(V_{j-1})$ and $y \in N(\La_j(t,z)\dots\La_s(t,z)^{p^{s-j}})$
such that
\bea
w+p^jy = p^{s+1}v-u,
\eea
hence  $u+w\in p^j \Z^r$. On the other hand, it follows from Lemma \ref{lem 1.2}
that $w \in N(\La_0)+ pN(\La_1) + \dots+p^{j-1}N(\La_{j-1})$, so that
\bea
u+w\in \Dl + N(\La_0)+ pN(\La_1) + \dots+p^{j-1}N(\La_{j-1}).
\eea
From the $\Dl$-admissibility we deduce that
$u+w = p^j \dl$ for some $\dl\in\Dl$, thus
$w=p^j \dl -\nobreak u$, \  $y=p^{s+1-j}v-\dl$ and
\begin{align*}
&
\Cf_{p^{s+1}v-u}\big(V_{j-1}(t,z)\La_j(t^{p^j},z^{p^j})\La_{j+1}(t^{p^j},z^{p^j})^p\dots\La_s(t^{p^j},z^{p^j})^{p^{s-j}}\big)
=
\\ &\qquad
=\sum_{\dl\in\Dl} \Cf_{p^j\dl-u}(V_{j-1}(t,z)) \cdot
\si^j\big(\Cf_{p^{s+1-j}v-\dl}\big(\La_j(t,z)\La_{j+1}(t,z)^p \dots\La_s(t,z)^{p^{s-j}}\big)\big).
\end{align*}
This proves  (ii).
\end{proof}

Our next results discuss congruences of the type 
$F_1(z)F_2(z)^{-1}\equiv G_1(z)G_2(z)^{-1}\pmod{p^s}$, where
 $F_1,F_2,G_1,G_2$ are $g\times g$ matrices whose entries are Laurent polynomials in $z$.
We consider such congruences when 
the determinants $\det F_2(z)$ and $\det G_2(z)$
are  Laurent polynomials  both nonzero  modulo~$p$.
Using Cramer's rule we write
  the entries of the inverse matrix $F_2(z)^{-1}$ in the 
 form $f_{ij}(z)/\det F_2(z)$ for $f_{ij}(z)\in\Z_p[z^{\pm1}]$ and do a similar computation for $G_2(z)$.
This presents the congruence $F_1(z)F_2(z)^{-1}\equiv G_1(z)G_2(z)^{-1}\pmod{p^s}$ 
in the form $F(z)\cdot1/\det F_2(z)\equiv G(z)\cdot1/\det G_2(z)\pmod{p^s}$ 
for some $g\times g$ matrices $F,G$ with entries in $\Z_p[z^{\pm1}]$,
while the latter is nothing else but the congruence 
$F(z)\cdot \det G_2(z)\equiv  G(z)\cdot \det F_2(z) \pmod{p^s}$.

\begin{thm}
\label{thm 1.6}
Let $(\La_0(t,z),\La_1(t,z),\dots,\La_l(t,z))$ be a $\Dl$-admissible
 tuple of Laurent polynomials in $\Z_p[x^{\pm1}]=\Z_p[t^{\pm1}, z^{\pm1}]$. 
\begin{enumerate}
\item[\textup{(i)}] For $0\leq s\leq l$ we have
\begin{align}
\notag
&
A\big(s+1, \La_0(x)\La_1(x)^{p}\dots \La_s(x)^{p^s}\big)\equiv
\\
\notag
&\qquad
\equiv A\big(1, \La_0(x)\big)\,
\si\big(A\big(1, \La_1(x)\big)\big) \dots \si^s\big(A\big(1, \La_s(x)\big)\big)
\pmod{p}.
\end{align}

\item[\textup{(ii)}] Assume that the determinants of  the matrices $ A\big(1, \La_i(t,z)\big)$, $i=0,1,\dots,l$, 
are Laurent polynomials all nonzero modulo~$p$.
Then for $1\leq s \leq l$ 
the determinant of the matrix $A\big(s, \La_1(x)\La_2(x)^p\dots \La_s(x)^{p^{s-1}} \big)$
is a Laurent polynomial nonzero modulo~$p$ and 
we have modulo $p^s$\,\textup:
\begin{align*}
&
A\big(s+1, \La_0(x)\La_1(x)^{p}\dots \La_s(x)^{p^s}\big)
\cdot
\si\big(A\big(s, \La_1(x)\La_2(x)^p\dots \La_s(x)^{p^{s-1}} \big)\big)^{-1}
\equiv
\\
&\qquad
\equiv
A\big(s, \La_0(x)\La_1(x)^{p}\dots \La_{s-1}(x)^{p^{s-1}}\big)
\cdot
\si\big(A\big(s-1, \La_1(x)\La_2(x)^{p}\dots \La_{s-1}(x)^{p^{s-2}}\big)\big)^{-1},
\end{align*}
where for $s=1$ we understand the second factor on the right
as the  $g\times g$ identity matrix.

\end{enumerate}

\end{thm}

\begin{proof}
By Lemma \ref{lem 1.5} we have
\begin{align*}
&
A\big(s+1, \La_0(x)\La_1(x)^{p}\dots \La_s(x)^{p^s}\big)
\equiv
\\
&\qquad
\equiv
A\big(1,\La_0(x)) \cdot  \si\big(A\big(s, \La_1(x)\La_2(x)^p\dots \La_s(x)^{p^{s-1}}\big)\bigr)
\pmod{p}.
\end{align*}
Iteration gives part (i) of the theorem.

\vsk.2>
If the determinants of  the matrices $ A\big(1, \La_i(t,z)\big)$, $i=0,1,\dots,l$, 
are Laurent polynomials all nonzero modulo~$p$, then part (i) implies that
the determinant
\bea
\det A\big(s, \La_1(x)\La_2(x)^p\dots \La_s(x)^{p^{s-1}} \big)
\equiv \prod_{j=1}^s \det \si^{j-1}\big(A\big(1, \La_j(t,z)\big)\big) \pmod{p},
\eea
is a Laurent polynomial nonzero modulo $p$.  
This proves the first statement of part (ii) of the theorem and allows us to consider the inverse matrices 
in the congruence of part (ii).

\vsk.2>

We prove part (ii) by induction on $s$. The case $s = 1$ follows from part (i).
For $1< s < l$ 
we substitute the expressions for $A(s+1, W_s(x))$ 
and $A(s, W_{s-1}(x))$
from part (ii) of Lemma~\ref{lem 1.5} into the two sides of the desired congruence:
\begin{align*}
&
A\big(s+1, \La_0(x)\La_1(x)^{p}\dots \La_s(x)^{p^s}\big)
\cdot
\si\big(A\big(s, \La_1(x)\La_2(x)^p\dots \La_s(x)^{p^{s-1}} \big)\big)^{-1}
=
\\ &\qquad
= A\big(1,V_0) 
+\sum_{j=2}^s A\big(j,V_{j-1}) \cdot \si^j\big(A\big(s-j+1, W_s^{(j)}\big)\big) \cdot  \si\big(A\big(s, W_s^{(1)}\big)\big)^{-1}
+
\\ &\qquad\quad
+ A\big(s+1,V_s) \cdot \si\big(A\big(s, W_s^{(1)}\big)\big)^{-1}
\\ \intertext{and}
&
A\big(s, \La_0(x)\La_1(x)^{p}\dots \La_{s-1}(x)^{p^{s-1}}\big)
\cdot
\si\big(A\big(s-1, \La_1(x)\La_2(x)^{p}\dots \La_{s-1}(x)^{p^{s-2}}\big)\big)^{-1}
=
\\ &\qquad
= A\big(1,V_0)
+\sum_{j=2}^s A\big(j,V_{j-1}) \cdot \si^j\big(A\big(s-j, W_{s-1}^{(j)}\big)\big) \cdot \si\big(A\big(s-1, W_{s-1}^{(1)}\big)\big)^{-1}.
\end{align*}
Since we want to compare these two expressions modulo $p^s$,
the last term in the upper sum containing
$A(s+1,V_s) \equiv 0 \pmod{p^s}$ can be ignored.

Given $j = 2, \dots , s$, we use the inductive hypothesis as
follows:
\begin{align*}
&
A\big(s-i+1, W_s^{(i)}\big)
\cdot
\si\big(A\big(s-i, W_s^{(i+1)} \big)\big)^{-1}
\equiv
\\ &\qquad
\equiv
A\big(s-i, W_{s-1}^{(i)}\big)
\cdot
\si\big(A\big(s-i-1, W_{s-1}^{(i+1)}\big)\big)^{-1}
\pmod{p^{s-i}}
\end{align*}
for $i=1,\dots,j-1$. Applying $\sigma^{i-1}$ to the $i$-th congruence and multiplying them out lead to telescoping products on both sides:
\begin{align*}
&
A\big(s, W_s^{(1)}\big)
\cdot
\si^{j-1}\big(A\big(s-j+1, W_s^{(j)}\big)\big)^{-1} \equiv
\\ &\qquad
\equiv
A\big(s-1, W_{s-1}^{(1)}\big)
\cdot
\si^{j-1}\big(A\big(s-j, W_{s-1}^{(j)}\big)\big)^{-1}
\pmod{p^{s-j+1}}.
\end{align*}
By our assumptions these four matrices are invertible.
Therefore, we can invert them to obtain the congruence
\begin{align}
\label{17}
&
\si^{j-1}\big(A\big(s-j+1, W_s^{(j)} \big)\big)
\cdot
A\big(s, W_s^{(1)}\big)^{-1}
\equiv
\\ &\qquad
\notag
\equiv
\si^{j-1}\big(A\big(s-j, W_{s-1}^{(j)}\big)\big)
\cdot
A\big(s-1, W_{s-1}^{(1)}\big)^{-1}
\pmod{p^{s-j+1}}.
\end{align}
Since $A(j,V_{j-1})\equiv 0\pmod{p^{j-1}}$, application of $\si$ to \eqref{17} and summation in $j$ of the resulted congruences
\begin{align*}
&
A(j,V_{j-1}(x))
\cdot
\si^j\big(A\big(s-j+1, W_s^{(j)} \big)\big)
\cdot
\si\big(A\big(s, W_s^{(1)}\big)\big)^{-1}\equiv
\\ &\qquad
\equiv
A(j,V_{j-1}(x))
\cdot
\si^j\big(A\big(s-j, W_{s-1}^{(j)}\big)\big)
\cdot
\si\big(A\big(s-1, W_{s-1}^{(1)}\big)\big)^{-1}
\pmod{p^{s}}
\end{align*}
completes the proof of part~(ii) of the theorem.
\end{proof}

\begin{cor}

Under the assumptions of part \textup{(ii)} of Theorem \textup{\ref{thm 1.6}}
for $1\leq s \leq l$ 
we have\,\textup:
\begin{align*}
&
\det A\big(s+1, \La_0(x)\La_1(x)^{p}\dots \La_s(x)^{p^s}\big)
\cdot
\det \si\big(A\big(s-1, \La_1(x)\La_2(x)^{p}\dots \La_{s-1}(x)^{p^{s-2}}\big)\big)
\equiv
\\
&
\equiv
\det A\big(s, \La_0(x)\La_1(x)^{p}\dots \La_{s-1}(x)^{p^{s-1}}\big)
\cdot
\det \si\big(A\big(s, \La_1(x)\La_2(x)^p\dots \La_s(x)^{p^{s-1}} \big)\big)
\!\!\!
\pmod{p^s}\,.
\end{align*}

\end{cor}

\subsection{Derivatives}

Recall that $z=(z_1,\dots,z_n)$. Denote
\bea
D_v=\frac{\der}{\der z_v}, \quad v=1,\dots,n.
\eea

\vsk.2>
Let $F_1(z),F_2(z),G_1(z),G_2(z) \in \Z_p[z^{\pm1}]$  and $m\geq 1$.
If
\bea
D_v(F_1(z))\cdot F_2(z)\equiv  D_v(G_1(z))\cdot G_2(z)\pmod{p^s}\,,
\eea
then
\begin{align}
\label{1.5}
&
D_v(\si^m(F_1(z)))\cdot\si^m(F_2(z)) - D_v(\si^m(G_1(z)))\cdot\si^m(G_2(z))
=
\\ &\qquad
= D_v(F_1(z^{p^m}))\cdot F_2(z^{p^m}) - D_v(G_1(z^{p^m}))\cdot G_2(z^{p^m})
=
\notag
\\ &\qquad
= p^mz_v^{p^m-1}\big(D_v(F_1(z))\cdot F_2(z) - D_v(G_1(z))\cdot G_2(z)\big)\big|_{z\to z^{p^m}}
\equiv
\notag
\\ &\qquad
\equiv 0 \pmod{p^{s+m}}.
\notag
\end{align}

\vsk.2>

\begin{thm}
\label{thm der}
Let $(\La_0(t,z),\La_1(t,z), \dots, \La_l(t,z))$ be a $\Dl$-admissible tuple of Laurent polynomials in $\Z_p[x^{\pm1}]=\Z_p[t^{\pm1},z^{\pm1}]$. Let $D=D_v$ for some $v=1,\dots,n$.
Then under assumptions  of part \textup{(ii)} of Theorem \textup{\ref{thm 1.6}} we have
\begin{align}
\label{Der}
&
D\big(\si^m\big(A\big(s+1, \La_0\La_1^p\dots \La_s^{p^s}\big)\big)\big)
\,\cdot\,
\si^m\big(A\big(s+1, \La_0\La_1^p\dots \La_s^{p^s}\big)\big)^{-1}
\equiv
\\ &\qquad
\notag
\equiv
D\big(\si^m\big(A\big(s, \La_0\La_1^p\dots \La_{s-1}^{p^{s-1}}\big)\big)\big)
\cdot
\si^m\big(A\big(s, \La_0\La_1^p\dots \La_{s-1}^{p^{s-1}}\big)\big)^{-1}
\pmod{p^{s+m}}
\end{align}
for all  $v=1,\dots,n$,  $1\leq s \leq l$  and  $m\geq 0$.

\end{thm}

\begin{proof}

First we notice that it is sufficient to establish the congruences \eqref{Der} for $m=0$, 
as the general $m$ case follows from \eqref{1.5}. So, we assume that $m=0$ and 
proceed by induction on $s\ge0$. For $s=0$  the statement is trivially true.

Consider the case of general $s$. Using part (ii) of Lemma \ref{lem 1.5} we can write
\begin{align*}
&
D(A(s+1, W_s))\,A(s+1, W_s)^{-1}
=
\\ &\qquad
=
\sum_{j=1}^{s+1} D(A(j,V_{j-1})) \, \si^j(A(s-j+1, W_s^{(j)})) \, A(s+1, W_s)^{-1}
+
\\ &\qquad\quad
+
\sum_{j=1}^{s+1} A(j,V_{j-1}) \,  D(\si^j(A(s-j+1, W_s^{(j)}))) \,A(s+1, W_s)^{-1}
\\ \intertext{and}
&
D(A(s, W_{s-1}))\, A(s, W_{s-1})^{-1}
=
\\ &\qquad
=
\sum_{j=1}^{s} D(A(j,V_{j-1})) \,  \si^j(A(s-j, W_{s-1}^{(j)})) \, A(s, W_{s-1})^{-1}
+
\\ &\qquad\quad
+
\sum_{j=1}^s A(j,V_{j-1}) \,  D(\si^j(A(s-j, W_{s-1}^{(j)}))) \, A(s, W_{s-1})^{-1}.
\end{align*}
The summand corresponding to $j=s+1$ in the first expression vanishes modulo $p^s$, because $A(s+1,V_s)\equiv0\pmod{p^s}$, implying that $D(A(s+1,V_s))\equiv0\pmod{p^s}$.
For the same reason $D(A(j,V_{j-1}))\equiv0\pmod{p^{j-1}}$ more generally;
combining this with the congruence 
\begin{align}
\label{17a}
&
\si^j(A(s-j+1, W_s^{(j)})) \, A(s+1, W_s)^{-1}
\equiv
\\ &\qquad
\notag
\equiv
\si^j(A(s-j, W_{s-1}^{(j)})) \, A(s, W_{s-1})^{-1}
\pmod{p^{s-j+1}}
\end{align}
and summing over $j$ we arrive at
\begin{align*}
&
\sum_{j=1}^{s+1} D(A(j,V_{j-1})) \, \si^j(A(s-j+1, W_s^{(j)})) \, A(s+1, W_s)^{-1}
\equiv
\\ &\qquad
\equiv
\sum_{j=1}^s D(A(j,V_{j-1})) \, \si^j(A(s-j, W_{s-1}^{(j)})) \, A(s, W_{s-1})^{-1}
\pmod{p^s}.
\end{align*}
Here \eqref{17a} follows from \eqref{17}, in which we take $j+1$ and $s+1$ for $j$ and $s$ and use 
$W_s=\La_0\La_1^p\dots\La_s^{p^s}$ instead of $W_{s+1}^{(1)}=\La_1\La_2^p\dots\La_{s+1}^{p^s}$.

To match the other sums we recall the inductive hypothesis in the form
\begin{align}
\label{18.1}
&
D(\si^j(A(s-j+1, W_s^{(j)}))) \, \si^j(A(s-j+1, W_s^{(j)}))^{-1}
\equiv
\\&\qquad
\notag
\equiv
D(\si^j(A(s-j, W_{s-1}^{(j)}))) \, \si^j(A(s-j, W_{s-1}^{(j)}))^{-1}
\pmod{p^s}
\end{align}
and notice that both sides in \eqref{18.1} are congruent to zero modulo $p^j$ by formula \eqref{1.5}. 
Therefore, multiplying congruences \eqref{18.1} and \eqref{17a} (in this order!) we obtain
\begin{align*}
&
D(\si^j(A(s-j+1, W_s^{(j)}))) \,A(s+1, W_s)^{-1}
\equiv
\\ 
&\qquad
\equiv
  D(\si^j(A(s-j, W_{s-1}^{(j)}))) \,A(s, W_{s-1})^{-1}
\pmod{p^s};
\end{align*}
then multiplying both sides of this congruence by  $A(j,V_{j-1})$ from the left and summing over $j$ we deduce
\begin{align*}
&
\sum_{j=1}^s A(j,V_{j-1}) \,  D(\si^j(A(s-j+1, W_s^{(j)}))) \,A(s+1, W_s)^{-1}
\equiv
\\ &\qquad
\equiv
\sum_{j=1}^s A(j,V_{j-1}) \,  D(\si^j(A(s-j, W_{s-1}^{(j)}))) \,A(s, W_{s-1})^{-1}
\pmod{p^s}.
\end{align*}
This completes the proof of the theorem.
\end{proof}

There are similar congruences for higher order derivatives of the matrices $A(s+1, W_s)$.
We restrict ourselves with the second order derivatives.

\begin{thm}
\label{thm der2}
Let $(\La_0(t,z),\La_1(t,z), \dots, \La_l(t,z))$ be a $\Dl$-admissible tuple of Laurent polynomials in 
$\Z_p[x^{\pm1}]=\Z_p[t^{\pm1},z^{\pm1}]$. 
Then under assumptions  of part \textup{(ii)} of Theorem \textup{\ref{thm 1.6}} we have
\begin{align}
\label{Der2}
&
D_u\big(D_v\big(A\big(s+1, \La_0\La_1^p\dots \La_s^{p^s}\big)\big)\big)
\,\cdot\,
A\big(s+1, \La_0\La_1^p\dots \La_s^{p^s}\big)^{-1}
\equiv
\\ &\qquad
\notag
\equiv
D_u\big(D_v\big(A\big(s, \La_0\La_1^p\dots \La_{s-1}^{p^{s-1}}\big)\big)\big)
\cdot
A\big(s, \La_0\La_1^p\dots \La_{s-1}^{p^{s-1}}\big)^{-1}
\pmod{p^{s+2m}}
\end{align}
for all $1\leq u, v\leq n$ and $s\geq 0$.
\end{thm}

\begin{proof}

Notice that, for an invertible  matrix $F(z)$  and a derivation $D$,
we have $D(F^{-1})=-F^{-1}\,D(F)\,F^{-1}$.

We apply the derivation $D_u$ to congruence \eqref{Der} with  $D=D_v$ and $m=0$:
\begin{align*}
&
D_u(D_v(A(s+1, \dots))) \, A(s+1, \dots)^{-1}
-
\\ &\quad
- D_v(A(s+1,\dots)) \, A(s+1, \dots)^{-1} \,
D_u(A(s+1, \dots)) \, A(s+1,\dots)^{-1}
\equiv
\\ &\;
\equiv
D_u(D_v(A(s,\dots))) \, A(s,\dots)^{-1}
-
\\ &\;\quad
- D_v(A(s, \dots)) \, A(s,\dots)^{-1} \,
D_u(A(s, \dots)) \, A(s,\dots)^{-1}
\pmod{p^s},
\end{align*}
where we write $A(s+1,\dots)$ and $A(s,\dots)$ for
\bea
A(s+1, \La_0\La_1^p\dots \La_s^{p^s}) \quad\text{and}\quad A(s, \La_0\La_1^p\dots \La_{s-1}^{p^{s-1}}).
\eea
It remains to apply \eqref{Der} with $D=D_u$ and $D=D_v$ and $m=0$
 to see that the second terms on both sides agree modulo~$p^s$.
After their cancellation we are left with the required congruences in~\eqref{Der2}.
\end{proof}

\section{Convergence}
\label{sec 3}

\subsection{Infinite tuples}

Let $\La=(\La_0(x), \La_1(x), \La_2(x),\dots )$ be an infinite $\Dl$-admissible tuple 
of Laurent polynomials in $\Z_p[x^{\pm1}]$ with only finitely many
distinct elements. Thus there is a finite set $\{F_1(x),\dots, F_f(x)\}$ $\subset$ 
$\Z_p[x^{\pm1}]=\Z_p[t^{\pm1},z^{\pm1}]$ of distinct Laurent polynomials
such that for any $j$ there is a unique $1\leq i_j\leq f$
with $\La_j(x) = F_{i_j}(x)$.

\begin{defn}
\label{def F}
The $\Dl$-admissible tuple $\La$ is called \emph{nondegenerate}, if
 for any $i=1,\dots,f$, the Laurent polynomial $\det A(1, F_i(x))\in
\Z_p[z^{\pm1}]$ is nonzero modulo~$p$. 

\end{defn}

Recall the notation:
\bea
W_s(x)=\La_0(x)\La_{1}(x)^{p}\dots \La_s(x)^{p^{s}},\quad
W_s^{(j)}(x)=\La_j(x)\La_{j+1}(x)^{p}\dots \La_s(x)^{p^{s-j}}.
\eea
If a $\Dl$-admissible tuple $\La$ is nondegenerate, then
for any $0\leq j\leq s$,  the Laurent polynomials $\det A(s-j+1,W_s^{(j)}(x))\in \Z_p[z^{\pm1}]$ 
are not congruent to zero modulo~$p$ and we may consider congruences involving the inverse matrices
$A(s-j+1,W_s^{(j)}(x))^{-1}$.

\subsection{Domain of convergence}

Recall that $z=(z_1,\dots,z_n)$. Denote

\bea
\frak D
= \{z\in \Z_p^n\ \mid \ |\det A(1,F_i(t,z))|_p=1, \,\,i=1,\dots,f\}.
\eea

\vsk.2>

\begin{lem}
\label{lem |det|}
For any $0\leq j\leq s$ and $a\in \frak D$ we have

\bea
\vert\det A(s-j+1,W_s^{(j)}(t,a))\big\vert_p =1.
\eea
\qed
\end{lem}

\begin{cor}
All entries of $A(s-j+1,W_s^{(j)}(t,z))^{-1}$ are
rational functions in $z$ regular on $\frak D$.  
For every $a\in\frak D$  all entries of $A(s-j+1,W_s^{(j)}(t,a))$ and $A(s-j+1,W_s^{(j)}(t,a))^{-1}$ are elements of $\Z_p$.
\qed

\end{cor}

\vsk.2>

\begin{thm}
\label{thm conv}
Let $\La=(\La_0(x), \La_1(x), \La_2(x),\dots )$ be an infinite nondegenerate
$\Dl$-admissible tuple.  Consider the sequence of $g\times g$ matrices 
\bean
\label{sec of m}
\Big(A\big(s+1, W_s(t,z)\big)\cdot \si\big(A\big(s, W_{s}^{(1)}(t,z)\big)\big)^{-1}\Big)_{s\geq 0}
\eean
whose entries are rational functions in $z$ regular on the domain $\frak D$.
This sequence uniformly converges on $\frak D$ as $s\to\infty$ to an analytic
$g\times g$ matrix with values in $\Z_p$. Denote this matrix by $\mc A_\La(z)$.
For $a\in\frak D$ we have
\bean
\label{det 1}
\big\vert \det \mc A_\La(a) \big\vert_p=1\,
\eean
and the matrix $\mc A_\La(a)$ is invertible. 
\end{thm}

\vsk.2>

\begin{proof}
By part (i) of Theorem \ref{thm 1.6} we have
$|\det\si\big(A\big(s, W_{s}^{(1)}(t,a)\big)\big)|_p=1$ for $a\in\frak D$.
Hence
$A\big(s+1, W_s(t,z)\big)\cdot \si\big(A\big(s, W_{s}^{(1)}(t,z)\big)\big)^{-1}$ 
is a matrix of rational functions
in $z$ regular on $\frak D$. Moreover, if $a\in\frak D$, then every entry of this matrix 
is an element of $\Z_p$.
The uniform convergence on $\frak D$ of the sequence \eqref{sec of m}
is a corollary of part (2) of Theorem \ref{thm 1.6}.
Equation \eqref{det 1} follows from part (i) of Theorem \ref{thm 1.6}.
The theorem is proved.
\end{proof}

\vsk.2>

\begin{thm}
\label{thm conv2}
Let $\La=(\La_0(x), \La_1(x), \La_2(x),\dots )$ be an infinite nondegenerate
$\Dl$-admissible tuple, and
$D=D_v$, $v=1,\dots,n$.
  Given $m\geq 0$ consider the sequence of $g\times g$ matrices 
\bea
\Big(D\big(\si^m\big(A\big(s+1, W_s(x)\big)\big)\big)\cdot \si^m\big(A\big(s+1, W_{s+1}(x)\big)\big)^{-1}\Big)_{s\geq 0}
\eea
whose entries are rational functions in $z$ regular on the domain $\frak D$.
This sequence uniformly converges on $\frak D$ as $s\to\infty$ to an analytic
$g\times g$ matrix with values in $\Z_p$. 
Denote this matrix  by $\mc A_{\La,D\si^m}(z)$.

\end{thm}

\begin{proof} 
The theorem is a corollary of Theorem \ref{thm der}.
\end{proof}

\begin{thm}
\label{thm conv3}
Let $\La=(\La_0(x), \La_1(x), \La_2(x),\dots )$ be an infinite nondegenerate
$\Dl$-admissible tuple.
  Given $1\leq u,v\leq k$, consider the sequence of $g\times g$ matrices 
\bea
\Big(D_u\big(D_v\big(A\big(s+1, W_s(x)\big)\big)\big)
\cdot A\big(s+1, W_{s+1}(x)\big)^{-1}\Big)_{s\geq 0}
\eea
whose entries are rational functions in $z$ regular on the domain $\frak D$.
This sequence uniformly converges on $\frak D$ as $s\to\infty$ to an analytic
$g\times g$ matrix with values in $\Z_p$. 
Denote this matrix  by $\mc A_{\La,D_uD_v}(z)$.

\end{thm}

\begin{proof} 
The theorem is a corollary of Theorem \ref{thm der2}.
\end{proof}

\vsk.2>

Let $\La=(\La_0(x), \La_1(x), \La_2(x),\dots )$ be an infinite nondegenerate
$\Dl$-admissible tuple.
 Let $1\leq u,v\leq n$.
 Consider the $g\times g$ matrix valued functions
$\mc A_{\La,\frac{\der}{\der z_u}\si^0}(z)$, $\mc A_{\La,\frac{\der}{\der z_v}\si^0}(z)$
in Theorem \ref{thm conv2} and denote them by $\mc A_u(z)$, $\mc A_v(z)$, respectively.
Consider the $g\times g$ matrix valued function
$\mc A_{\La,\frac{\der}{\der z_u}\frac{\der}{\der z_v}}(z)$ in Theorem \ref{thm conv3}
and denote it by $\mc A_{u,v}(z)$.
All the three functions are analytic on $\frak D$.

\begin{lem}
\label{lem conv3}
We have
\bea
\frac{\der}{\der z_u}\mc A_v = \mc A_{u,v} - \mc A_v\mc A_u\,.
\eea

\end{lem}

\begin{proof}
The lemma follows from the formula
\begin{equation*}
\frac{\der}{\der z_u} \Big(\frac {\der A}{\der z_v}\cdot A^{-1}\Big) =
\frac {\der^2 A}{\der z_u\der z_v}\cdot A^{-1}
- \frac {\der A}{\der z_v}\cdot A^{-1} \cdot \frac {\der A}{\der z_v}\cdot A^{-1}.
\qedhere
\end{equation*}
\end{proof}

\section{KZ equations}
\label{sec 4}

\subsection{KZ equations}
Let $\g$ be a simple Lie algebra with an invariant scalar product.
The { Casimir element}  is  $\Om = {\sum}_i \,h_i\ox h_i  \in \ \g \ox \g$,
where $(h_i)\subset\g$ is an orthonormal basis.
Let  $V=\otimes_{i=1}^n V_i$ be 
a tensor product of $\g$-modules, $\ka\in\C^\times$ a nonzero number.
The {\it  KZ equations} is the system of differential 
equations on a $V$-valued function $I(z_1,\dots,z_n)$,
\bea
\frac{\der I}{\der z_i}\ =\ \frac 1\ka\,{\sum}_{j\ne i}\, \frac{\Om_{i,j}}{z_i-z_j} I, \qquad i=1,\dots,n,
\eea
where $\Om_{i,j}:V\to V$ is the Casimir operator acting in the $i$th and $j$th tensor factors,
see \cite{KZ, EFK}.

\vsk.2>

This system is a system of Fuchsian first order
 linear differential equations. 
  The equations are defined on the complement in $\C^n$ to the union of all diagonal hyperplanes.
 
\vsk.2>

The object of our discussion is the following particular case.  Let $n=2g+1$ be an odd positive integer.
 We consider
 the following system of differential and algebraic  equations
for a column $n$-vector $I=(I_1,\dots,I_n)$ depending on variables $z=(z_1,\dots,z_n)$\,:
\bean
\label{KZ}
\phantom{aaa}
 \frac{\partial I}{\partial z_i} \ = \
   {\frac 12} \sum_{j \ne i}
   \frac{\Omega_{ij}}{z_i - z_j}  I ,
\quad i = 1, \dots , n,
\qquad
I_1+\dots+I_{n}=0,
\eean
where $z=(z_1,\dots,z_n)$,
the $n\times n$-matrices $\Om_{ij}$ have the form
\bea
 \Omega_{ij} \ = \ \begin{pmatrix}
             & \vdots^{\kern-1.2mm i} &  & \vdots^{\kern-1.2mm j} &  \\
        {\scriptstyle i} \cdots & {-1} & \cdots &
            1   & \cdots \\
                   & \vdots &  & \vdots &   \\
        {\scriptstyle j} \cdots & 1 & \cdots & -1&
                 \cdots \\
                   & \vdots &  & \vdots &
                   \end{pmatrix} ,
\eea
and all other entries are zero.
 This  joint system of {\it differential and 
algebraic equations} will be called the {\it system of KZ  equations} in this paper.

For $i=1,\dots,n$ denote
\bean
\label{GH}
&
H_i(z) =   {\frac 12} \sum_{j\ne i}    \frac{\Omega_{ij}}{z_i - z_j}\,,
\qquad
\nabla_i^{\on{KZ}} = \frac{\der}{\der z_i} - H_i(z), \qquad i=1,\dots,n.
\eean
The linear operators $H_i(z)$ are called the Gaudin Hamiltonians.
The KZ equations can be written as the system of equations,
\bea
\nabla_i^{\on{KZ}}I=0, \quad i=1,\dots,n,\qquad I_1+\dots + I_n =0.
\eea
\vsk.2>

System  \eqref{KZ} is the system of the  differential KZ equations with 
parameter $\ka=2$ associated with the Lie algebra $\sll_2$ and the subspace of singular vectors of weight $2g-1$ of the tensor power 
$(\C^2)^{\ox {(2g+1)}}$ of two-dimensional irreducible $\sll_2$-modules, up to a gauge transformation, see 
this example in  \cite[Section 1.1]{V2}.

\subsection{Solutions over $\C$}
\label{sec 11.4}

Define the {\it master function}
\bea
\Phi(t,z) = (t-z_1)^{-1/2}\dots (t-z_n)^{-1/2}
\eea
and the column $n$-vector
\bean
\label{KZ sol} 
I^{(C)}(z) = (I_1,\dots,I_n):=
\int_{C}
\Big(\frac {\Phi(t,z)}{t-z_1}, \dots , \frac {\Phi(t,z)}{t-z_n}\Big)dt
\,,
\eean
where  $C\subset \C-\{z_1,\dots,z_n\}$  
is a contour on which the integrand  takes its initial value when $t$ encircles $C$.

\begin{thm}[{cf.~\cite{V5}}]
The function $I^{(C)}(z)$ is a solution of system \eqref{KZ}.

\end{thm}

This theorem is a very particular case of the results in \cite{SV1}.

\begin{proof}  
The theorem follows from  Stokes' theorem and the two identities:
\bean
\label{i1}
-\frac 12\,
\Big(\frac {\Phi(t,z)}{t-z_1}
 + \dots  + \frac {\Phi(t,z)}{t-z_n}\Big)\,  
=\, \frac{\der\Phi}{\der t}(t,z)\,,
\eean
\bean
\label{i2}
\Big(\frac{\der }{\der z_i}-\frac12
\sum_{j\ne i} \frac {\Omega_{i,j}}{z_i-z_j} \Big)
\Big(\frac {\Phi(t,z)}{t-z_1}, \dots, \frac {\Phi(t,z)}{t-z_n}\Big)\,  
= \frac{\der \Psi^i}{\der t} (t,z),
\eean
where  $\Psi^i(t,z)$ is the column $n$-vector   $(0,\dots,0,-\frac{\Phi(t,z)}{t-z_i},0,\dots,0)$ with 
the nonzero element at the $i$-th place. 
\end{proof}

\begin{thm} [{cf.~\cite[Formula (1.3)]{V1}}]
\label{thm dim}

All solutions of system \eqref{KZ} have this form. 
Namely, the complex vector space of solutions of the form \eqref{KZ sol} is $(n-1)$-dimensional.

\end{thm}

\subsection{Solutions as vectors of first derivatives}
\label{sec 11.5}

Consider the hyperelliptic integral
\bea
T(z) = T^{(C)}(z) =
\int_C
\Phi(t,z) \,dt.
\eea
Then
\bea
I^{(C)}(z) 
=
\,
2\,
\Big(\frac {\der T^{(C)}}{\der z_1}, \dots ,
\frac {\der T^{(C)}}{\der z_n}\Big).
\eea
Denote
$\nabla T =
\Big(\frac {\der T}{\der z_1},\dots,  \frac {\der T}{\der z_n}\Big)$.
Then the column gradient vector of the function $T(z)$ satisfies the following system of  (KZ) equations
\bea
\nabla_i^{\on{KZ}} \nabla T =0, \quad i=1,\dots,n,\qquad  
\frac {\der T}{\der z_1} +\dots + \frac {\der T}{\der z_n}=0.
\eea
This is a system of second order linear differential equations on the function $T(z)$.

\subsection{Solutions modulo $p^s$}
\label{sec:new}

For an integer $s\geq 1$ define the {\it master polynomial}
\bea
\Phi_s(t,z) = \big((t-z_1)\dots(t-z_n)\big)^{(p^s-1)/2}.
\eea
Recall that $n=2g+1$.   For $\ell=1,\dots, g$ define the column $n$-vector
\bea
I_{s,\ell} (z)=(I_{s,\ell,1}, \dots, I_{s,\ell.n})
\eea
as the coefficient of $t^{\ell p^s-1}$ in the column $n$-vector of polynomials 
\bea
\Big(\frac{\Phi_s(t,z)}{t-z_1}, \dots, \frac {\Phi_s(t,z)}{t-z_n}\Big).
\eea

Notice that $\deg_t \frac{\Phi_s(t,z)}{t-z_i} = (2g+1)\frac{p^s-1}2-1$. If
$\ell \notin \{1,\dots,g\}$\,,\ then $\frac{\Phi_s(t,z)}{t-z_i}$ does not have the monomial 
$t^{\ell p^s-1}$.

\vsk.2>

\begin{thm} [\cite{V5}]
\label{thm 7.3}
The column $n$-vector  $I_{s,\ell}(z)$ of polynomials in $z$ is a solution of system \eqref{KZ}
modulo $p^s$.

\end{thm}

\vsk.2>

The column $n$-vectors  $I_{s,\ell}(z)$, $\ell=1,\dots,g$, were
called in \cite{V5} the {\it $p^s$-hypergeometric solutions} of the KZ equations
\eqref{KZ}.

\vsk.2>
\begin{proof}

We have the following modifications of identities \eqref{i1}, \eqref{i2}\,:
\bea
\frac {p^s-1}2\,
\Big(\frac {\Phi_s(t,z)}{t-z_1}
 + \dots + \frac {\Phi_s(t,z)}{t-z_n}\Big)\,  
=\, \frac{\der\Phi_s}{\der t}(t,z)\,,
\eea
\bea
\Big(\frac{\der }{\der z_i} + \frac {p^s-1}2
\sum_{j\ne i} \frac {\Omega_{i,j}}{z_i-z_j} \Big)
\Big(\frac {\Phi_s(t,z)}{t-z_1}, \dots, \frac {\Phi_s(t,z)}{t-z_n}\Big)\,  
= \frac{\der \Psi_s^i}{\der t} (t,z),
\eea
where  $\Psi_s^i(t,z)$ is the column $n$-vector   $(0,\dots,0,-\frac{\Phi_s(t,z)}{t-z_i},0,\dots,0)$ with 
the nonzero element at the $i$-th place. Theorem \ref{thm 7.3} follows from these identities.
\end{proof}

Consider the $n\times g$ matrix
\bea
I_s(z) = (I_{s,1},\dots,I_{s,g}) = \big(I_{s, \ell, i}\big)_{\ell = 1,\dots,g}^{i=1,\dots,n}\ ,
\eea
where $I_{s, \ell, i}$ stays at the $\ell$-th column and $i$-th row. The matrix $I_s(z)$ satisfies the KZ equations,
\bea
\nabla_i^{\on{KZ}}I_s(z) =0, \quad i=1,\dots,n,\qquad I_{s,\ell,1}+\dots + I_{s,\ell,n}(z) =0, 
\quad \ell=1,\dots,g,
\eea
modulo $p^s$.

\section{Congruences for solutions of KZ equations}
\label{sec 5}
\subsection{Congruences for Hasse--Witt matrices of KZ equations}
Let $p> 2g+1$,
\bean
\label{DN}
\Dl =\{1,\dots,g\}\subset \Z, \qquad N=[0, gp + (p-1)/2-g]\subset \R.
\eean

\begin{lem}
The infinite tuple $(N, N,\dots)$ is $\Dl$-admissible\textup, see Definition \textup{\ref{defN}}.
\qed

\end{lem}

Denote
\bea
F(t,z) : = \Phi_1(t,z) = \big((t-z_1)\dots(t-z_n)\big)^{(p-1)/2}.
\eea
 The Newton polytope of $F(t,z)$ with respect to variable $t$ is the interval
 \linebreak
$N=[0, gp + (p-1)/2 -g]$, see \eqref{DN}, and
\bea
\Phi_s(t,z) = F(t,z)\cdot F(t,z)^p\dots F(t,z)^{p^{s-1}}\,.
\eea
The infinite tuple $(F(t,z), F(t,z),\dots)$ is $\Dl$-admissible, see Definition \ref{defn}.

\vsk.2>
For $s\geq 1$ consider the Hasse--Witt $g\times g$ matrix 
\bea
A(s, \Phi_s(t,z))
=
\big( \Cf_{p^{s}v-u}(\Phi_s(t,z))\big)_{u,v=1,\dots,g}\,,
\eea
see \eqref{Cuv}. The entries of this matrix are polynomials in $z$.

\vsk.2>

\begin{thm} 
\label{thm F ne}

The polynomial $\det A(1, F(t,z))$ is nonzero  modulo $p$.

\end{thm}

\begin{proof}

Consider the lexicographical ordering of monomials
$z_1^{d_1}\dots z_{2g+1}^{d_{2g+1}}$. We  have $z_1>\dots > z_{2g+1}$ and  so on.
For a nonzero Laurent polynomial
$f(z)=\sum_{d_1,\dots,d_{2g+1}} a_{d_1,\dots,d_{2g+1}} z_1^{d_1}$
\dots $z_{2g+1}^{d_{2g+1}}$\, with coefficients in $\Z$\,,\
the nonzero summand 
$a_{d_1,\dots,d_{2g+1}} z_1^{d_1}\dots z_{2g+1}^{d_{2g+1}}$ with the largest monomial
$z_1^{d_1}$ \dots $z_{2g+1}^{d_{2g+1}}$ is called
 the { leading term} of $f(z)$.
 
If $f(z)$ and $g(z)$ are two nonzero Laurent polynomials, then
the leading term of $f(z)g(z)$ equals the product of the leading terms of $f(z)$ and $g(z)$.

\vsk.2>

Denote $A(1, F(t,z)) = : (A_{u,v}(z))_{u,v=1,\dots,g}$\,.

\begin{lem}
If $p>2g+1$, the leading term of $A_{u,v}(z)$ equals 
\bea
&&
\pm\binom{(p-1)/2}{v-u} 
(z_1z_2\dots z_{2g+1-2v})^{(p-1)/2}/ z_{2g+1-2v}^{v-u}, \qquad \text{if}\;\; v\geq u,
\\
\notag
&&
\pm\binom{(p-1)/2}{u-v} 
(z_1z_2\dots z_{2g+1-2v})^{(p-1)/2} z_{2g+2-2v}^{u-v}, \qquad \quad\ \text{if}\;\; v\leq u.
\eea
For example, for $g=2$ the matrix of leading terms is
\bean
\label{ex d}
\begin{pmatrix}
 \pm(z_1z_2z_3)^{(p-1)/2} & \pm \binom{(p-1)/2}{1}z_1^{(p-1)/2}/z_1
 \\
   \pm \binom{(p-1)/2}{1}(z_1z_2z_3)^{(p-1)/2}z_4 & \pm z_1^{(p-1)/2} 
\end{pmatrix} .
\eean
\end{lem}

\begin{proof}
The proof is by inspection.
\end{proof}

It is easy to see that the leading term of the determinant of the matrix of leading terms
of $A_{u,v}(z)$ equals the product of diagonal elements,
\bean
\label{Deg}
\pm\ \prod_{v=1}^{g}( z_1\dots z_{2g+1-2v})^{(p-1)/2}.
\eean
This term is not congruent to zero modulo $p$.
This proves Theorem \ref{thm F ne}.
\end{proof}

\begin{cor}
\label{cor F ne}
The $\Dl$-admissible infinite tuple $(F(t,z), F(t,z),\dots )$ satisfies the assumptions of Theorem \textup{\ref{thm 1.6}}.
Therefore,
\begin{enumerate}
\item[\textup{(i)}] for $s\geq 1$ we have
\begin{align}
\label{alal+}
A(s,  \Phi_{s}(t,z))
\equiv
A(1, F(t,z))\cdot\si(A(1, F(t,z)))\dots \si^{s-1}(A\big(1, F(t,z)))
\pmod{p}\,;
\end{align}
\item[\textup{(ii)}] for $s\geq 1$ 
the determinant of the matrix $A(s,  \Phi_s(t,z))$
is a polynomial, which is nonzero modulo~$p$, and 
we have modulo $p^s$\,\textup:
\begin{align*}
&
A(s+1,  \Phi_{s+1}(t,z))\cdot \si(A(s,  \Phi_{s}(t,z)))^{-1}
\equiv
A(s,  \Phi_{s}(t,z))\cdot \si(A(s-1,  \Phi_{s-1}(t,z)))^{-1},
\end{align*}
where for $s=1$ we understand the second factor on the right 
as the  $g\times g$ identity matrix.
\end{enumerate}

\end{cor}

\begin{proof}
The corollary follows from Theorems \ref{thm F ne} and \ref{thm 1.6}.
\end{proof}

\vsk.2>

\subsection{Congruences for frames of solutions of KZ equations}

\begin{thm}
\label{thm coS}
We have the following congruences of $n\times g$ matrices.

\begin{enumerate}
\item[\textup{(i)}] For $s\geq 1$,
\bea
I_{s+1}(z) \cdot A(s+1,  \Phi_{s+1}(t,z))^{-1}
\equiv 
I_{s}(z) \cdot A(s,  \Phi_{s}(t,z))^{-1}
\pmod{p^s}\,.
\eea

\item[\textup{(ii)}] For $s\geq 1$ and $j=1,\dots, n$,
\bea
\frac {\der I_{s+1}}{\der z_j} (z) \cdot A(s+1,  \Phi_{s+1}(t,z))^{-1}
\equiv 
\frac{\der I_{s}}{\der z_j}(z) \cdot A(s,  \Phi_{s}(t,z))^{-1}
\pmod{p^s}\,.
\eea

\end{enumerate}

\end{thm}

\vsk.2>

\begin{proof}
Consider the first row of the Hasse--Witt matrix $A(s, \Phi_s(t,z))$,
\bea
\big(A_{1,1}(s, \Phi_s(t,z)),\dots, A_{1,g}(s, \Phi_s(t,z))\big),
\quad
A_{1,\ell}(s, \Phi_s(t,z)) = \on{Coeff}_{\ell p^s-1}(\Phi_{s}(t,z)).
\eea
For each  $A_{1,\ell}(s, \Phi_s(t,z))$ we view the gradient
\bea
\nabla A_{1,\ell}(s, \Phi_s(t,z))=\Big(\frac{\der A_{1,\ell}(s)}{\der z_1}, \dots,
\frac{\der A_{1,\ell}(s)}{\der z_n} \Big)
\eea
as a column $n$-vector. The resulting $n\times g$ matrix of gradients
\bea
\nabla A(s,z):=(\nabla A_{1,1}(s, \Phi_s(t,z)),\dots,\nabla A_{1,g}(s, \Phi_s(t,z)))
\eea
is proportion to the matrix $I_s(z)$, $\nabla A(s,z) = \frac{1-p^s}2 I_s(z)$.
By Theorems \ref{thm der} and \ref{thm der2} we have
modulo $p^s$,
\bea
&
\nabla A(s+1,z)
\cdot A(s+1,  \Phi_{s+1}(t,z))^{-1}
\equiv 
\nabla A(s,z) \cdot A(s,  \Phi_{s}(t,z))^{-1},
\\
&
\frac{\der}{\der z_j}\big(\nabla A(s+1,z)\big)
\cdot A(s+1,  \Phi_{s+1}(t,z))^{-1}
\equiv 
\frac{\der}{\der z_j}\big(\nabla A(s,z)\big) \cdot A(s,  \Phi_{s}(t,z))^{-1}.
\eea
These congruences imply the theorem.
\end{proof}

\begin{cor}
\label{thm KZ mod p}
For $s\geq 1$ we have
\bea
I_{s}(z) \cdot A(s,  \Phi_{s}(t,z))^{-1} \equiv 
I_{1}(z) \cdot A(1,  \Phi_{1}(t,z))^{-1} \pmod{p}.
\eea

\end{cor}

\section{Convergence of solutions of KZ equations}
\label{sec 6}

\subsection{Nonzero polynomials}
\label{sec ax}

\begin{lem}
\label{lem nonempty}

Let $z=(z_1,\dots,z_n)$.
Let $F(z)\in  \F_p[z]$
be a nonzero polynomial, $\deg F(z) \leq d$ for some $d$.
Let $p^m>d$. Then there are at least
$\frac{p^{mn}-1}{p^m-1} (p^m-1-d) + 1 $ points
of\  $(\F_{p^m})^n$ where $F(z)$ is nonzero.

\end{lem}

\begin{proof}
First we show that there exists
$a\in (\F_{p^m})^n$ such that $F(a) \ne 0$.
The  proof is by induction on $n$. If $n=1$, then
the nonzero polynomial $F(z_1)$ cannot have more than $d$ zeros. Hence
there exists $a\in \F_{p^m}$ such that $F(a) \ne 0$.

Assume that the existence of $a$ is proved for all nonzero polynomials with less than $n$
variables.  Write $F(z)=\sum_i c_i(z_2,\dots,z_{n}) z_1^i$.
By the induction assumption,   there exists $a_2,\dots,a_n \in \F_{p^m}$ such that
$c_i(a_2,\dots,a_{n})\ne 0$ for at least one $i$. 
Hence 
$F(z_1,a_2\dots,a_{n})$ is a nonzero polynomial of degree $\leq d$
which defines a nonzero function of $z_1$. The existence of $a$ is proved.

Let $a\in (\F_{p^m})^n$ be such that
$F(a) \ne 0$. In $(\F_{p^m})^n$ there are  $\frac{p^{mn}-1}{p^m-1}$ distinct lines 
through $a$. On each of the lines there are at least $p^m-1-d$ points different from $a$ where $F(z)$
is nonzero.  Hence the total number of points where $F(z)\ne 0$ is at least
$\frac{p^{mn}-1}{p^m-1} (p^m-1-d) + 1$.
\end{proof}

\subsection{Unramified extensions of $\Q_p$}

We fix  an algebraic closure $\overline{\Q_p}$ of $\Q_p$.
For every $m$, there is a unique unramified extension of $\Q_p$ in 
 $\overline{\Q_p}$ of degree $m$, denoted by $\Q_p^{(m)}$.
This can be obtained by attaching to $\Q_p$ a primitive root of $1$ of order $p^m-1$.
The norm $|\cdot|_p$ on $\Q_p$ extends to a norm $|\cdot|_p$ on 
$\Q_p^{(m)}$.
Let 
\bea
\Z_p^{(m)} = \{ a\in \Q_p^{(m)} \mid |a|_p\leq 1\}
\eea
denote the ring of integers in $\Q_p^{(m)}$. The ring $\Z_p^{(m)}$
has the unique maximal ideal 
\bea
\mathbb M_p^{(m)} = \{ a\in \Q_p^{(m)} \mid |a|_p <1\},
\eea
such that $\mathbb Z_p^{(m)}\big/ \mathbb M_p^{(m)}$ is isomorphic to the finite field
$\F_{p^m}$.

For every $u\in\F_{p^m}$ there is a unique $\tilde u\in \mathbb Z_p^{(m)}$ that is a lift of $u$ and such that 
$\tilde u^{p^m}=\tilde u$. The element $\tilde u$ is called the Teichmuller lift of $u$.

\subsection{Domain $\frak D_B$}
For $u\in\F_{p^m}$ and $r>0$ denote
\bea
D_{u,r} = \{ a\in \Z_p^{(m)}\mid |a-\tilde u|_p<r\}\,.
\eea
We have the partition
\bea
\Z_p^{(m)} = \bigcup_{u\in\F_{p^m}} D_{u,1}\,.
\eea

Recall  $z=(z_1,\dots,z_n)$. For $B(z) \in \Z[z]$, define
\bea
\frak D_B \ =\ \{ a\in (\Z_p^{(m)})^n\,  \mid \  |B(a)|_p=1\} .
\eea
Let $\bar B(z)$ be the projection of $B(z)$ to $\F_p[z]\subset \F_{p^m}[z]$.
Then $\frak D_B$ is the union of unit polydiscs,
\bea
\frak D_B = \bigcup_{\substack{u_1,\dots,u_n\in \F_{p^m}\\  \bar B(u_1,\dots, u_n)\ne 0}} \  D_{u_1,1}\times \dots
\times D_{u_n,1}\,.
\eea
For any $k$ we have
\begin{align}
\notag
\{ a\in (\Z_p^{(m)})^n \mid \ |B(a^{p^k})|_p=1\}
&=\bigcup_{\substack{u_1,\dots,u_n\in \F_{p^m}\\  \si^k(\bar B(u_1,\dots, u_n))\ne 0}} \  D_{u_1,1}\times \dots
\times D_{u_n,1} =
\\
\notag
&=\bigcup_{\substack{u_1,\dots,u_n\in F_{p^m}\\  \bar B(u_1,\dots, u_n)\ne 0}} \  D_{u_1,1}\times \dots
\times D_{u_n,1} =
 \frak D_B \,.
\end{align}

\subsection{Uniqueness theorem}

Let  $\frak D \subset (\Z_p^{(m)})^n$ be the union of some of the
unit polydiscs
\linebreak
$ D_{u_1,1}\times \dots \times D_{u_n,1}$\,, where $u_1,\dots,u_n\in\F_{p^m}$.

Let $(F_i(z))_{i=1}^\infty$ and $(G_i(z))_{i=1}^\infty$ be two sequences of rational functions
on $(\F_{p^m})^n$. Assume that each of the rational functions has the form 
$P(z)/Q(z)$, where $P(z), Q(z)\in\Z[z]$, and for any
polydisc $ D_{u_1,1}\times \dots \times D_{u_n,1}\,\subset \frak D$, we have
\bea
|Q(\tilde u_1,\dots,\tilde u_n)|_p=1,
\eea
which implies that
\bea
|Q(a_1,\dots,a_n)|_p=1,\qquad \forall\ (a_1,\dots,a_n)\in \frak D.
\eea
Assume that  the sequences 
$(F_i(z))_{i=1}^\infty$ and $(G_i(z))_{i=1}^\infty$
 uniformly converge on $\frak D$ to analytic functions, which we denote by $F(z)$ and $G(z)$, respectively. 

\begin{thm}
\label{thm U}
Under these assumptions, if $F(z)=G(z)$ on an open 
nonempty subset of~$\frak D$. Then $F(z)=G(z)$ on~$\frak D$.

\end{thm}

The following proof was communicated to us by Vladimir Berkovich.

\begin{proof}

The domain $\frak D$ is a disjoint union of open unit polydiscs, and so it gives rise to a similar domain 
$\frak D’$ over the algebraic closure of $\Q_p^{(m)}$. 
Each rational function of our sequence   has no poles in $\frak D’$. 
This property implies that the restriction of the function to each of open unit
polydisc contained in $\frak D$ is a formal power series convergent at all points of the polydisc. 

First, recall the definition and some properties of the affine space $\mathbb A^n$
over a non-Archime\-dean field $\mathbb K$ (as $\Q_p^{(m)}$). As a space it is the set of 
all multiplicative seminorms $| \,\cdot\, |_x\colon \mathbb K[T_1,\allowbreak\dots,T_n] \to \R_+$ 
that extend the (multiplicative) valuation $|\,\cdot\,|\colon \mathbb K\to \R_+$, 
and it is provided with the weakest topology with respect to which all 
functions $\mathbb A^n\to \R_+\colon x \mapsto |f|_x$ for polynomials
 $f$ are continuous. We need only one point $g$, called the Gauss point 
 and defined as follows: $|\sum_\mu a_\mu T^\mu|_g = \max_\mu |a_\mu|$. One can show that
\begin{enumerate}
\item[(1)] the point $g$ lies in the closure of each open polydisc
 of radius one with center at a point $t\in \mathbb K^n$ with $|t_i|\le 1$, and 
\item[(2)] for each bounded convergent power series $f$ on such an open polydisc, the real valued function $x \mapsto |f|_x$ extends by continuity to the point $g$, and one has $|f|_x \le |f|_g$ for all points of the polydisc. 
\end{enumerate}
{\sl Uniqueness}: Since $F$ and $G$ are uniform limits of rational functions regular on $\frak D$, 
their restrictions to each open polydisc in $\frak D$  are bounded convergent power series and, 
in particular, the number $|F-G|_g$ is well defined and one has $|F-G|_x \le |F-G|_g$ for 
all points $x\in \frak D$. If $F(x)=G(x)$ for points from a nonempty open subset of an open
unit polydisc, 
then $F(x)=G(x)$ for all points of the polydisc (it is uniqueness property for convergent power series) 
and, therefore, $|F-G|_g=0$. This implies that $F=G$ on $\frak D$.
\end{proof}

\subsection{Domain of convergence}

By Theorem \ref{thm F ne} the polynomial $\det A(1, F(t,z))\in \Z[z]$ is  nonzero modulo $p$.
The polynomial $\det A(1, F(t,z))$ is a homogeneous polynomial in $z$ of degree
\bean
\label{deg d}
d = \frac{p-1}2 g^2\,,
\eean
cf.~\eqref{Deg}.
 Define
\bea
\frak D^{(m)}_{\on{KZ}}
= \{a \in (\Z_p^{(m)})^{n} \mid\  |\det A(1, F(t,a))|_p=1\}\,.
\eea
By Lemma \ref{lem nonempty} the domain
$\frak D^{(m)}_{\on{KZ}}$ is nonempty if $p^m> d$.
In what follows we assume that $p^m>d$.

\begin{rem}
The space $(\Z_{p^m})^n$ is the disjoint union of $p^{mn}$ unit polydiscs
$ D_{u_1,1}\times \dots \times D_{u_n,1}$.  By Lemma \ref{lem nonempty}
 at least $\frac{p^{mn}-1}{p^m-1} (p^m-1-d) + 1>p^{mn}\big(1-\frac d{p^m-1} \big)$ 
 of them belong to  $\frak D^{(m)}_{\on{KZ}}$.
So, as $m$ grows almost all polydiscs belong to  $\frak D^{(m)}_{\on{KZ}}$.

\end{rem}

We have
$\vert\det A(s, \Phi_{s}(t,a))\big\vert_p =1$ for $ a\in \frak D^{(m)}_{\on{KZ}}$.
All entries of $A(s, \Phi_{s}(t,z))^{-1}$ are
rational functions in $z$ regular on $\frak D^{(m)}_{\on{KZ}}$.  
For every $a\in\frak D^{(m)}_{\on{KZ}}$  all entries of $A(s, \Phi_{s}(t,a))$ and $A(s, \Phi_{s}(t,a))^{-1}$ 
are elements of $\Z_p^{(m)}$.

\begin{thm}
\label{thm coKZ}
The sequence of $g\times g$ matrices 
\bea
\big(A\big(s+1,  \Phi_{s}(t,z)\big)\cdot \si\big(A\big(s,  \Phi_{s-1}(t,z)\big)\big)^{-1}\big)_{s\geq 1}\,,
\eea
whose entries are rational functions in $z$ regular on  $\frak D^{(m)}_{\on{KZ}}$,
 uniformly converges on $\frak D^{(m)}_{\on{KZ}}$ as $s\to\infty$ to an analytic
$g\times g$ matrix which will be denoted by $\mc A(z)$.
For $a\in\frak D^{(m)}_{\on{KZ}}$ we have
\bea
\big\vert \det \mc A(a) \big\vert_p=1\,
\eea
and the matrix $\mc A(a)$ is invertible. 
\end{thm}

\begin{proof}
The theorem follows from Theorem \ref{thm conv}.
\end{proof}

\begin{thm}
\label{thm coKZ}
For $i=1,\dots,n$ the sequence of $g\times g$ matrices 
\bea
\Big(\Big(\frac{\der}{\der z_i}A\big(s,  \Phi_{s}(t,z)\big)\Big)\cdot A\big(s,  \Phi_{s}(t,z)\big)^{-1}\Big)_{s\geq 1}\,,
\eea
whose entries are rational functions in $z$ regular on  $\frak D^{(m)}_{\on{KZ}}$,
uniformly converges on $\frak D^{(m)}_{\on{KZ}}$ as $s\to\infty$ to an analytic
$g\times g$ matrix, which will be denoted by $\mc A^{(i)}(z)$.

The sequence of $n\times g$ matrices 
\bea
\big(I_s(z)\cdot A\big(s,  \Phi_{s}(t,z)\big)^{-1}\big)_{s\geq 1}\,,
\eea
whose entries are rational functions in $z$ regular on  $\frak D^{(m)}_{\on{KZ}}$,
uniformly converges on $\frak D^{(m)}_{\on{KZ}}$ as $s\to\infty$ to an analytic
$n\times g$ matrix which will be denoted by $\mc I(z)$.

For $i=1,\dots, n$ the sequence of $n\times g$ matrices 
\bea
\Big(\frac{\der I_s}{\der z_i}(z)\cdot A\big(s,  \Phi_{s}(t,z)\big)^{-1}\Big)_{s\geq 1}\,,
\eea
whose entries are rational functions in $z$ regular on  $\frak D^{(m)}_{\on{KZ}}$,
uniformly converges on $\frak D^{(m)}_{\on{KZ}}$ as $s\to\infty$ to an analytic
$n\times g$ matrix which will be denoted by $\mc I^{(i)}(z)$.

We have
\bea
\frac{\der \mc I}{\der z_i} = \mc I^{(i)} - \mc I \cdot \mc A^{(i)}\,.
\eea

\end{thm}

\begin{proof}
The theorem follows from Theorems \ref{thm conv2},  \ref{thm conv3}, and Lemma \ref{lem conv3}.
\end{proof}

\begin{thm}
\label{thm KZ mc}

We have the following system of equations on $\frak D^{(m)}_{\on{KZ}}$\,:
\bea
\mc I^{(i)} = H_i \cdot \mc I, \qquad i=1,\dots, n,
\eea
where $H_i$ are the Gaudin Hamiltonians defined in \eqref{GH}.

\end{thm}

\begin{proof}

The theorem is a corollary of Theorem \ref{thm 7.3}.
\end{proof}

\begin{cor}
\label{cor mc I 1}

For $a\in \frak D^{(m)}_{\on{KZ}}$ we have
\bea
\mc I(a) \equiv I_1(a)\cdot A\big(1,  \Phi_{1}(t,a)\big)^{-1}
\pmod{p}.
\eea

\end{cor}

\begin{proof} 

The corollary follows from Corollary \ref{thm KZ mod p}
 and Theorem \ref{thm coKZ}.
\end{proof}

\subsection{Vector bundle $\mc L  \,\to\,  \frak D^{(m),o}_{\on{KZ}}$} 

Denote 
\bea
W=\{(I_1,\dots,I_n)\in (\Q_p^{(m)})^n\ |\ I_1+\dots+I_n=0\}.
\eea
  We consider vectors
$(I_1,\dots,I_n)$ as column vectors. 
The differential operators $\nabla^{\on{KZ}}_i$, $i=1,\dots,n$, define a connection on
the trivial bundle $W\times \frak D^{(m)}_{\on{KZ}} \to \frak D^{(m)}_{\on{KZ}}$,
called the KZ connection. The connection has singularities at the diagonal hyperplanes in $(\Z_p^{(m)})^n$
 and is well-defined over 

\bea
\frak D^{(m),o}_{\on{KZ}} = \{ a=(a_1,\dots,a_n)\in(\Z_p^{(m)})^n \mid |\det A(1,F(t,a))|_p=1, \, a_i\ne a_j\,\ \forall i,j\}. 
\eea

\vsk.2>
\noindent
The KZ connection is flat, 
\bea
\big[\nabla^{\on{KZ}}_i, \nabla^{\on{KZ}}_j\big]=0 \qquad \forall\,i,j,
\eea
see \cite{EFK}.
The flat sections of the KZ connection are solutions of  system  \eqref{KZ} of KZ equations.
 
\vsk.2>
For any $a\in \frak D^{(m)}_{\on{KZ}}$ let $\mc L_a \subset W$ be the vector subspace generated by columns of the 
$n\times g$ matrix $\mc I(a)$. Then
\bea
\mc L := \bigcup\nolimits_{a\in  \frak D^{(m)}_{\on{KZ}}}\,\mc L_a \,\to\,  \frak D^{(m)}_{\on{KZ}}
\eea
is an analytic distribution of vector subspaces  in the fibers of the trivial bundle
$W\times \frak D^{(m)}_{\on{KZ}} \to \frak D^{(m)}_{\on{KZ}}$.

\begin{thm}
\label{thm inv} 

The distribution $\mc L  \,\to\,  \frak D^{(m)}_{\on{KZ}}$   is invariant with respect to the KZ connection.
In other words, if $s(z)$ is a local section of $\mc L  \,\to\,  \frak D^{(m)}_{\on{KZ}}$, then the
sections $\nabla_i^{\on{KZ}} s(z)$, $i=1,\dots,n$, also are sections of $\mc L  \,\to\,  \frak D^{(m)}_{\on{KZ}}$.

\end{thm}

\begin{proof} 
Let $\mc I(z)= (\mc I_1(z), \dots, \mc I_g(z))$ be columns of the $n\times g$ matrix 
$\mc I(z)$.
Let $a\in \frak D^{(m)}_{\on{KZ}}$. Let $c(z) = (c_1(z),\dots,c_g(z))$ be a column vector of analytic functions at $a$.
Consider a local section of the distribution  $\mc L  \,\to\,  \frak D^{(m)}_{\on{KZ}}$,
\bea
s(z)\ =\  \sum_{j=1}^g \,c_j(z) \mc I_j(z) \  = : \ \mc I \cdot c .
\eea
Then
\bea
\nabla^{\on{KZ}}_i s(z) 
&=&
 - H_i\cdot \mc I\cdot c +\frac{\der \mc I}{\der z_i} \cdot c 
+ \mc I \cdot \frac{\der c}{\der z_i}
\\
&=&
 - H_i\cdot \mc I\cdot c +( \mc I^{(i)} - \mc I \cdot \mc A^{(i)})\cdot c 
+ \mc I \cdot \frac{\der c}{\der z_i}
\\
&=&
 - H_i\cdot \mc I\cdot c +( H_i \cdot \mc I  - \mc I \cdot \mc A^{(i)})\cdot c 
+ \mc I \cdot \frac{\der c}{\der z_i}
\\
&=&
  - \mc I \cdot \mc A^{(i)} \cdot c 
+ \mc I \cdot \frac{\der c}{\der z_i} \,.
\eea
Clearly, the last expression is a local section of $\mc L  \,\to\,  \frak D^{(m)}_{\on{KZ}}$.
\end{proof}

\begin{thm}
\label{thm rk}

The function $a \mapsto \dim_{\Q_p^{(m)}} \mc L_a$ is  constant on
$\frak D^{(m), o}_{\on{KZ}}$, in other words,  
$\mc L  \,\to\,  \frak D^{(m)}_{\on{KZ}}$  is a vector bundle over $\frak D^{(m),o}_{\on{KZ}}\subset \frak D^{(m)}_{\on{KZ}}$.

\end{thm}

\begin{proof}

First, we prove that  the  function $a \mapsto \dim_{\Q_p^{(m)}} \mc L_a$ is  locally constant on
$\frak D^{(m),o}_{\on{KZ}}$. This  holds true in the following more general  setting.
Let $k$ be a positive integer,  $a\in (\Z_p^{(m)})^n$. Let 
$B_i(z)$, $i=1,\dots,n$, be $k\times k$ matrices
defined and  analytic in a neighborhood of $a$.
The differential operators $\mc B_i = \frac{\der}{\der z_i} - B_i(z)$,
$i=1,\dots,n$,\, 
act on $(\Q_p^{(m)})^k$-valued functions defined and analytic in a neighborhood of $a$.
The  operators  $(\mc B_i)$ define a connection $\nabla^{\mc B}$ on the restriction of the trivial bundle 
$(\Q_p^{(m)})^k\times (\Q_p^{(m)})^n\to (\Q_p^{(m)})^n$ to  a neighborhood of $a$.  Assume that the connection is flat,
$[\mc B_i,\mc B_j]=0$ for all $i,j$.  Then, for a sufficiently  small neighborhood $D$ of $a$, the
space $\mc S$ of solutions of the system
$\mc B_i y =0$, $i=1,\dots,n$, on $D$ is a $k$-dimensional $\Q_p^{(m)}$-vector space.
For any $b\in D$ the values of solutions at $b$ span $(\Q_p^{(m)})^k$.
Under these assumptions,\,  the $\nabla^{\mc B}$-invariant  subspace distributions in fibers of 
$(\Q_p^{(m)})^k\times (\Z_p^{(m)})^n\to (\Z_p^{(m)})^n$ over $D$ are labeled by
$\Q_p$-vector subspaces $Y\subset \mc S$. The corresponding  distribution assigns to $b\in D$
the subspace $\{ y(b)\mid y\in Y\}\subset (\Q_p^{(m)})^k$. Such distributions have constant rank.
Hence the  function $a \mapsto \dim_{\Q_p^{(m)}} \mc L_a$ is  locally constant on
$\frak D^{(m),o}_{\on{KZ}}$. 

By Theorem \ref{thm U} the locally constant function $a \mapsto \dim_{\Q_p^{(m)}} \mc L_a$  cannot 
take more than one value on $\frak D^{(m),o}_{\on{KZ}}$ since  the dimension 
$\dim_{\Q_p^{(m)}} \mc L_a$ may drop from its maximal value only on 
a proper analytic subset of $\frak D^{(m)}_{\on{KZ}}$. The theorem is proved.
\end{proof}

Recall that $d$ is the degree of the polynomial $\det A(1, F(t,z))$, see \eqref{deg d}.

\begin{thm}
\label{thm rk g}

If $p^m> 2d$, then the  analytic vector bundle $\mc L\to \frak D^{(m),o}_{\on{KZ}}$ is of rank $g$.

\end{thm}

\begin{proof}
First we show that there is a $g\times g$ minor of the $n\times g$ matrix valued function
$\mc I(z)$ which is nonzero on $\frak D^{(m)}_{\on{KZ}}$. By Corollary \ref{cor mc I 1} this fact holds true  
if there is a $g\times g$ minor of the $n\times g$ matrix 
$I_1(z)\cdot A\big(1,  \Phi_{1}(t,z)\big)^{-1}$,  which defines a function
on $\frak D^{(m)}_{\on{KZ}}$ nonzero modulo $p$. 
Since   $|\det A\big(1,  \Phi_{1}(t,a)\big)|_p=1$ for any $a\in \frak D^{(m)}_{\on{KZ}}$,
 it is enough to prove that there is a nonzero $g\times g$  minor of the $n\times g$ polynomial matrix
$I_1(z)$. But this fact follows from \cite[Lemma 7.2]{V5}, also see \cite[Lemma 6.1]{V4}. 
More precisely, \cite[Lemma 7.2]{V5} implies that the 
leading term of the $g\times g$ minor in rows $1,3,\dots,2g-1$ equals 
\bea
\pm \prod_{l=1}^g\binom{(p-1)/2}{l} z_1^{(p-1)/2}\dots z_{2g-2l}^{(p-1)/2}z_{2g-2l+1}^{{(p-1)/2}-l}\,.
\eea
The degree of this minor equals
\bea
g(2g+1)\frac{p-1}2 - p(1+\dots+ g) = g^2\frac{p-1}2-\frac{ g(g+1)}2 \ <\  g^2\frac{p-1}2 =d.
\eea
Thus, we have two polynomials of degree $\leq d$: this minor and 
$\det A\big(1,  \Phi_{1}(t,z)\big)$. Both polynomials are nonzero modulo $p$. 
By Lemma \ref{lem nonempty} if $p^m> 2d$, then this minor is nonzero on 
$\frak D^{(m),o}_{\on{KZ}}$.
\end{proof}

\subsection{Remarks}
\subsubsection{}
It was shown in \cite[Section 10.4]{V5} that the span of columns of the
$n\times g$ matrix $I_s(z)$ has a $p$-adic limit as $s\to\infty$ when $z$ belongs to one of the asymptotic zones
of the KZ equations. The limit is a $g$-dimensional  space of power series solutions of the KZ equations with respect to
the coordinates attached to that asymptotic zone.
It is not clear yet if that asymptotic zone belongs to $\frak D^{(m)}_{\on{KZ}}$.

\subsubsection{} One may expect that the subbundle $\mc L\to \frak D^{(m),o}_{\on{KZ}}$
can be extended to a rank $g$ subbundle over 
$\frak D^{(m)}_{\on{KZ}} - \frak D^{(m),o}_{\on{KZ}}$, the union
 of the diagonal hyperplanes in 
$ \frak D^{(m)}_{\on{KZ}}$.

 \subsubsection{}
 Following Dwork we may expect that locally at any point $a\in\frak \frak D^{(m),o}_{\on{KZ}}$, the
 solutions of the KZ equations with values in $\mc L\to \frak D^{(m),o}_{\on{KZ}}$
 are given at $a$ by power series in 
$z_i-a_i$, $i=1,\dots,n$,  bounded in their polydiscs of convergence, 
while any other local solution at $a$ is given by a power series unbounded in its polydisc of convergence,
cf. \cite{Dw} and \cite[Theorem A.4]{V5}.

\subsubsection{}

The KZ connection $\nabla^{\on{KZ}}_i$, $i=1,\dots,n$, over $\C$
is identified with the Gauss--Manin connection of the family
of hyperelliptic curves $y^2=(t-z_1)\dots (t-z_n)$. The monodromy representation of that
 Gauss--Manin connection is described in \cite[Appendix]{Ch}. The image of the monodromy representation is
 so big that the connection does not have proper invariant subbundles.
 Thus the existence of the invariant subbundle $\mc L\to \frak D^{(m),o}_{\on{KZ}}$
 is a $p$-adic feature.

\subsubsection{}
The invariant subbundles of the KZ connection over $\C$ usually are related to some additional conformal block
constructions, for example see \cite{FSV, SV2, V3}. Apparently our subbundle  $\mc L\to \frak D^{(m),o}_{\on{KZ}}$
 is of a
different $p$-adic nature.

\end{document}